\documentclass[final,3p,times]{elsarticle}

\usepackage[T1]{fontenc}
\usepackage[utf8]{inputenc}
\usepackage{microtype}
\usepackage{amsmath,amssymb,mathtools,amsthm}
\usepackage{enumitem}
\usepackage{xcolor}
\usepackage[colorlinks=true,linkcolor=blue!45!black,citecolor=blue!45!black,urlcolor=blue!45!black]{hyperref}

\biboptions{sort&compress}
\allowdisplaybreaks
\setlist[enumerate]{leftmargin=2.2em}
\setlist[itemize]{leftmargin=2em}

\newtheorem{theorem}{Theorem}[section]
\newtheorem{proposition}[theorem]{Proposition}
\newtheorem{lemma}[theorem]{Lemma}
\newtheorem{corollary}[theorem]{Corollary}
\theoremstyle{definition}
\newtheorem{definition}[theorem]{Definition}
\theoremstyle{remark}
\newtheorem{remark}[theorem]{Remark}
\newtheorem{question}{Problem}

\newcommand{\FS}{\mathbf{FS}}
\newcommand{\RB}{\mathbf{RB}}
\newcommand{\Fin}{\operatorname{Fin}}
\newcommand{\id}{\operatorname{id}}
\newcommand{\ev}{\operatorname{ev}}
\newcommand{\const}{\operatorname{const}}
\newcommand{\Up}{\mathord{\Uparrow}}
\newcommand{\Down}{\mathord{\Downarrow}}
\newcommand{\Disc}{\mathsf{Disc}}

\newcommand{\up}{\mathord{\uparrow}}
\newcommand{\down}{\mathord{\downarrow}}

\journal{Preprint}

\begin{document}

\begin{frontmatter}

\title{A function space characterization of RB-domains}

\author[addr1]{Yuxu Chen}
\address[addr1]{School of Mathematics, Sichuan University, Chengdu, P.R. China 610065}
\ead{chenyuxu@scu.edu.cn}

\begin{abstract}
We prove that a domain $D$ is an RB-domain if and only if the function space
$[D\to E]$ is continuous for every domain $E$, thereby resolving a conjecture
of Luan and Li.  
More precisely, for each domain $D$ we construct an explicit algebraic test domain $A_D$ such that continuity of the single function space $[D\to A_D]$ already forces $D$ to be an RB-domain.  
For countably based domains $D$, the test domain can be chosen as a fixed domain that is independent of $D$.
The analogous continuity characterization of FS-domains is false: the closed-disk domain is an FS-domain, but its function space into a suitable pointed algebraic domain is not continuous.  
\end{abstract}


\end{frontmatter}

\section{Introduction}

For domains $D$ and $E$, the pointwise ordered set $[D\to E]$ of
Scott-continuous maps is always a dcpo, but it need not be continuous even
when both $D$ and $E$ are continuous.  This is a basic obstruction to using
all continuous domains as a higher-order semantic category
\cite[Section~4.3]{AbramskyJung1994}.  A central theme in domain theory is
therefore to identify natural classes whose approximation structure is
preserved by products and function spaces.  Plotkin's bifinite domains give
the classical algebraic example \cite{Plotkin1976,Smyth1983}.  In the
continuous setting, Jung introduced FS-domains and proved their Cartesian
closure \cite{Jung1989,Jung1990}; see also
\cite[Section~4.2]{AbramskyJung1994}.  RB-domains, equivalently the
Scott-continuous retracts of bifinite domains, form a subclass of the
FS-domains \cite[Definition~2.2 and Lemma~2.4]{ZouLiGuo2018}.

The main problem considered here is whether an RB-domain can
be recognized from its outgoing function spaces.  Miao, Luan, Liu and Li
proved that $[D\to E]$ is continuous whenever $D$ is an RB-domain and $E$ is
a domain \cite{MiaoEtAl2023}.  Luan and Li asked whether the converse
holds and gave a positive answer when $D$ is algebraic, in which case
RB-domains are precisely the bifinite domains~\cite[Section~4, Theorem~4.6]{LuanLi2025}.  
The case of arbitrary continuous domains was left open.  Our main result, stated and proved as
Theorem~\ref{thm:rb-characterization}, answers this question.  It shows that, for every domain $D$,
\[
 D\in\RB
 \quad\Longleftrightarrow\quad
 [D\to E]\text{ is continuous for every domain }E.
\]
More precisely, from the order and the way-below relation of $D$ we construct
an explicit algebraic domain $A_D$ such that continuity of the single
function space $[D\to A_D]$ already implies that $D$ is an RB-domain.

For countably based cases, the test target can be chosen independently of
the source of the RB-domains.  Theorem~\ref{thm:uniform-countable-test} constructs a fixed
countably based algebraic domain $A_\omega$ such that, for every countably
based domain $D$,
\[
 D\in\RB
 \quad\Longleftrightarrow\quad
 [D\to A_\omega]\text{ is continuous}.
\]
For pointed countably based domains, the same role is played by a fixed
pointed countably based algebraic domain $U_\omega$.

The analogous continuity statement for FS-domains is false, since the latest result shows that $\Disc$, 
Lawson's planar closed-disk domain~\cite{AbramskyJung1994, Lawson2008} is an FS-domain but not an RB-domain~\cite{ChenKouLyu2026}. Therefore $[\Disc \to E_\Disc]$ is not conitnuous. We give out concrete calculation to show that how the constructed pointed algebraic target $E_\Disc$ destroys the continuity of $[\Disc \to E_\Disc]$.  
The proof combines the finite-support mechanism above with the finite-image Lorentz-flux theorem in \cite[Proposition~5.3]{ChenKouLyu2026}.

\section{Preliminaries}\label{sec:preliminaries}

We use standard domain-theoretic terminology as in
\cite{AbramskyJung1994,GierzEtAl2003}.  A
\emph{domain} is a nonempty continuous dcpo, and directed subsets are
understood to be nonempty.  For a set $A$, let $\Fin(A)$ denote the finite
subsets of $A$, including $\varnothing$.  We write
\(
  \Up x=\{y:x\ll y\},
  \Down x=\{y:y\ll x\}.
\)
A subset $B$ of a domain $X$ is a basis if, for every $x\in X$, the set
$B\cap\Down x$ is directed with supremum $x$.  An element $k$ of a dcpo is compact if $k\ll k$.  
A dcpo is algebraic if every element is the supremum of the compact elements below it. A dcpo is pointed if it has a least element.
A domain is countably based if it has a countable basis.  For dcpos $X$ and $Y$, the function space
$[X\to Y]$ is the dcpo of
Scott-continuous maps, ordered pointwise; directed suprema are computed
pointwise.  The notation $\ll_{\mathcal O(X)}$ denotes the way-below relation
in the lattice of Scott-open subsets of $X$.

A directed family $\mathcal F\subseteq[X\to X]$ with
$\bigvee\mathcal F=\id_X$ is called an approximate identity. Note that we use the unpointed convention and no idempotence condition is imposed on a deflation.

\begin{definition}\cite{AbramskyJung1994}\label{def:fs-rb}
Let $X$ be a domain. 
\begin{enumerate}[label=\textup{(\alph*)}]
\item A Scott-continuous map $f:X\to X$ is finitely separated from
$\id_X$ if there is a finite set $M\subseteq X$ such that, for every
$x\in X$, some $m\in M$ satisfies
\(
  f(x)\leq m\leq x.
\)
The domain $X$ is an \textit{FS-domain} if it has an approximate identity consisting
of maps finitely separated from the identity.
\item A deflation on $X$ is a Scott-continuous map $d:X\to X$ with finite
image and $d\leq\id_X$.  The domain $X$ is an \textit{RB-domain} if it has an
approximate identity consisting of deflations.
\end{enumerate}
\end{definition}

 Every deflation is finitely separated by its image, and hence $\RB\subseteq\FS$. 
A pointed domain $L$ is an \emph{$L$-domain} if every principal ideal
$\down x$ is a complete lattice.  A domain is \emph{coherent} if the
intersection of any two Scott-compact saturated sets is Scott compact.
Every pointed domain is compact in its Scott topology.  Consequently,
for pointed domains, coherence is equivalent to Lawson compactness
\cite[Proposition~4.2.20]{AbramskyJung1994}.  
We also use the following standard results.
\begin{proposition}\label{prop:fact}
  \begin{enumerate}[label=\textup{(\roman*)}]
\item If $D$ is an RB-domain and $E$ is a domain, then $[D\to E]$ is
continuous \cite{MiaoEtAl2023}.
\item If $D$ is a coherent domain and $E$ is a pointed RB-domain, then
$[D\to E]$ is an RB-domain
\cite[Corollary~4.9]{XiXuLawson2016}.  
\item If $D$ is a domain, $E$ is a pointed domain, and $[D\to E]$ is
continuous, then either $D$ is coherent or $E$ is an $L$-domain
\cite[Lemma~4.3.1]{AbramskyJung1994}.
\item If $D$ and $E$ are arbitrary domains and $[D\to E]$ is continuous,
then either $D$ is compact in its Scott topology or $E$ is a disjoint union
of pointed domains \cite[Lemma~4.3.6]{AbramskyJung1994}.
\end{enumerate}
\end{proposition}

For a domain $E$, being a disjoint union of pointed domains is equivalent to
every element of $E$ lying above a unique minimal element.

\section{Function-space characterization of RB-domains}\label{sec:rb}

Miao et al.~\cite{MiaoEtAl2023} proved that if $D$ is an RB-domain and
$E$ is any domain, then the function space $[D\to E]$ is continuous.
Luan and Li~\cite{LuanLi2025} established the converse for algebraic
domains and left open whether the algebraicity assumption can be removed.
We will show that the converse is indeed true.
Fix a domain $X$, we construct an algebraic test domain $A_X$.  
The continuity of $[X\to A_X]$ provides the finite
approximation needed to construct finite-image deflations on $X$.

\subsection{The algebraic test domain}

Let
\(
 \mathcal{D}(X)
 =\{\varnothing\}\cup
   \{A\subseteq X:A\text{ is a directed subset of }X\}
\) and let $\mathcal{P}(X)$ be the powerset of $X$.
Define
\[
 E_X=
 \bigl(\mathcal{D}(X)\times\{0\}\bigr)
 \cup
 \bigl(\mathcal P(X)\times\{1\}\bigr),
\]
ordered componentwise:
\(
 (A,i)\leq(B,j)
 \Longleftrightarrow
 A\subseteq B\text{ and }i\leq j.
\)
When the second coordinate is $0$, the first coordinate must be directed or
empty.  When the second coordinate is $1$, the first coordinate may be an
arbitrary subset of $X$.

\begin{lemma}\label{lem:EX-algebraic}
The poset $E_X$ is a pointed algebraic domain. Its least and greatest
elements are $(\varnothing,0)$ and $(X,1)$, respectively.  The following
compact elements form a basis:
\[
 \mathcal K_X={}
 \{(F,0):F=\varnothing\text{ or }F\subseteq X
                    \text{ is finite and directed}\}
 \cup\{(M,1):M\in\Fin(X)\}.
\]
\end{lemma}

\begin{proof}
Let $((A_i,\varepsilon_i))_{i\in I}$ be a directed family in $E_X$.
If $\varepsilon_i=0$ for every $i$, then $\bigcup_iA_i$ is empty or
directed.  Indeed, if $a\in A_i$ and $b\in A_j$, choose $k\in I$ above
$i$ and $j$.  Then $a,b\in A_k$, and the directedness of $A_k$ supplies
$c\in A_k$ with $a,b\leq c$.  If some $\varepsilon_i$ is $1$, the first
coordinate of the supremum need not be directed.  In either case the
supremum is
\[
 \bigvee_{i\in I}(A_i,\varepsilon_i)
 =
 \begin{cases}
  (\bigcup_iA_i,0),&\varepsilon_i=0\text{ for every }i,\\[1mm]
  (\bigcup_iA_i,1),&\varepsilon_i=1\text{ for some }i.
 \end{cases}
\]
The same description also shows that $(\varnothing,0)$ and $(X,1)$ are,
respectively, the least and greatest elements of $E_X$.

We next verify compactness.  The least element $(\varnothing,0)$ is
compact.  Suppose
$(F,0)\leq\bigvee_i(A_i,\varepsilon_i)$, where $F$ is nonempty, finite and
directed.  For every $a\in F$, choose $i_a$ with $a\in A_{i_a}$; since
$F$ is finite and the indexing family is directed, one index $i$ lies above
all the $i_a$.  Then $F\subseteq A_i$, and hence
$(F,0)\leq(A_i,\varepsilon_i)$.  If
$(M,1)\leq\bigvee_i(A_i,\varepsilon_i)$, choose an index $i_0$ whose
second coordinate is $1$, and for each $m\in M$ choose $i_m$ with
$m\in A_{i_m}$.  A common upper index $i$ satisfies
$M\subseteq A_i$ and $\varepsilon_i=1$, so
$(M,1)\leq(A_i,1)$.  Thus every element of $\mathcal K_X$ is compact.

Let $A\subseteq X$ be directed.  Its finite directed subsets form a
directed family under inclusion.  Indeed, if $F$ and $G$ are two such
subsets, directedness of $A$ gives an element $a\in A$ above the finite set
$F\cup G$, and $F\cup G\cup\{a\}$ is a finite directed subset of $A$
containing both.  Their union is $A$, and therefore
\(
 (A,0)=\bigvee\{(F,0):F\subseteq A\text{ is finite and directed}\}.
\)
The case $A=\varnothing$ is immediate.  When the second coordinate is $1$, the finite subsets
of an arbitrary $A\subseteq X$ are directed by finite union and have union
$A$, so
\(
 (A,1)=\bigvee\{(M,1):M\subseteq A\text{ is finite}\}.
\)
Hence every element is the supremum of compact elements below it, and
$E_X$ is algebraic.
\end{proof}

Define $Q:X\to E_X$ by
$Q(x)=(\Down x,0).$
Thus the first coordinate of $Q(x)$ is the directed set of all way-below
approximants of $x$.

\begin{lemma}\label{lem:Q-general}
The map $Q$ is Scott continuous.
\end{lemma}

\begin{proof}
If $x\leq y$, then $\Down x\subseteq\Down y$, so $Q$ is monotone.  Let
$D\subseteq X$ be directed.  The inclusion
\(
 \bigcup_{x\in D}\Down x\subseteq\Down\bigl(\bigvee D\bigr)
\)
is immediate.  Conversely, if $u\ll\bigvee D$, interpolation gives
$w$ with $u\ll w\ll\bigvee D$.  Since $w\ll\bigvee D$, there is
$x\in D$ with $w\leq x$, and then $u\ll x$.  Hence
\(
 \Down\bigl(\bigvee D\bigr)=\bigcup_{x\in D}\Down x.
\)
All values $Q(x)$ have second coordinate $0$, so the directed-supremum formula in
Lemma~\ref{lem:EX-algebraic} now gives
$Q(\bigvee D)=\bigvee_{x\in D}Q(x)$.  Thus $Q$ is Scott continuous.
\end{proof}

We use the standard step function of Erker--Escard\'o--Keimel~\cite{ErkerEscardoKeimel1998}.  Let $Y$ be a pointed dcpo, let $k$ be compact in $Y$, let $f:X\to Y$ be Scott continuous,
and let $V$ be Scott open.  Define
\[
 s_{V,k}(x)=
 \begin{cases}
  k,&x\in V,\\
  \bot_Y,&x\notin V.
 \end{cases}
\]
Since $V$ is Scott open, the map $s_{V,k}$ is Scott continuous.

\begin{lemma}\label{lem:standard-step}
Let $X$ be a dcpo and $Y$ a pointed dcpo.  Let $V\in\mathcal O(X)$,
let $k$ be a compact element of $Y$, and let $f\in[X\to Y]$.
If
\(
V\ll_{\mathcal O(X)}f^{-1}(\up k),
\)
then $s_{V,k}\ll f$ in $[X\to Y]$.
\end{lemma}

\begin{proof}
Let $(g_i)_{i\in I}$ be a directed family in $[X\to Y]$ such that
$f\leq \bigvee_{i\in I}g_i$.  For each $i\in I$, put
\(
W_i=g_i^{-1}(\up k).
\)
Since $k$ is compact, $\up k$ is Scott open, so each $W_i$ is
Scott open.  Moreover, the family $(W_i)_{i\in I}$ is directed under
inclusion.

We claim that
\(
f^{-1}(\up k)\subseteq\bigcup_{i\in I}W_i.
\)
Indeed, if $x\in f^{-1}(\up k)$, then
\(
k\leq f(x)\leq\bigvee_{i\in I}g_i(x).
\)
Since $k$ is compact and $(g_i(x))_{i\in I}$ is directed, there exists
$i\in I$ such that $k\leq g_i(x)$.  Hence $x\in W_i$.

Because
\(
V\ll_{\mathcal O(X)}f^{-1}(\up k),
\)
there exists $i\in I$ such that $V\subseteq W_i$.  Thus
$s_{V,k}(x)\leq g_i(x)$ for every $x\in X$: if $x\in V$, then
$k\leq g_i(x)$, while if $x\notin V$, then
$\bot_Y\leq g_i(x)$.  Therefore $s_{V,k}\leq g_i$, and hence
$s_{V,k}\ll f$.
\end{proof}

For $u\ll v$ in $X$, we have
$\Up v\subseteq\up v\subseteq\Up u$.  The principal upset
$\up v$ is Scott compact.  Hence every directed Scott-open cover whose
union contains $\Up u$ has one member containing $\up v$, and therefore
one member containing $\Up v$.  Thus
\(\Up v\ll_{\mathcal O(X)}\Up u\).
Define
\[
 S_{u,v}(x)=
 \begin{cases}
  (\{u\},0),&v\ll x,\\
  (\varnothing,0),&v\not\ll x.
 \end{cases}
\]
The element $(\{u\},0)$ is compact in $E_X$, and
\(
 Q^{-1}\bigl(\up(\{u\},0)\bigr)= \Up u.
\)
This yields the following consequence.

\begin{lemma}\label{lem:requirements-waybelow}
Let $X$ be a domain, and let $E_X$, $Q:X\to E_X$, and
$S_{u,v}:X\to E_X$ be defined as above.  If $u,v\in X$ satisfy
$u\ll v$, then \(S_{u,v}\ll Q\) in the function space $[X\to E_X]$.
\end{lemma}

\begin{proof}
We have
$\Up v\ll_{\mathcal O(X)} \Up u$ and
$Q^{-1}(\up(\{u\},0))=\Up u$.  Since $(\{u\},0)$ is compact in
$E_X$, Lemma~\ref{lem:standard-step} applied to
$V=\Up v$, $k=(\{u\},0)$ and $f=Q$ yields
$S_{u,v}\ll Q$.
\end{proof}

For $M\in\Fin(X)$, let $K_M:X\to E_X$ be the constant function with value
$(M,1)$.  The family $(K_M)_{M\in\Fin(X)}$ is directed and
\( \bigvee_{M\in\Fin(X)}K_M \)
is the constant function with value $(X,1)$.  Since $Q$ lies below this
constant top function, every $\varphi\ll Q$ satisfies
$\varphi\leq K_M$
for some finite $M\subseteq X$.
The following lemma turns this finite directed information into a deflation.

\begin{lemma}\label{lem:deflation-extraction}
Let $X$ be a domain, let $M$ be a finite subset of $X$, and let
$\varphi:X\to E_X$ be Scott continuous. Suppose that
$\varphi\leq Q$ and $\varphi\leq K_M$. For each $x\in X$, write
$\varphi(x)=(B_\varphi(x),0)$. Then
\(
B_\varphi(x)\in\mathcal D(X)
\ \text{and}\
B_\varphi(x)\subseteq M\cap\Down x.
\)
If $B_\varphi(x)$ is nonempty for every $x\in X$, then
\(
d_\varphi(x)=\max B_\varphi(x)
\)
defines a deflation on $X$ whose image is contained in $M$.
Moreover, if $\psi:X\to E_X$ satisfies the same assumptions and
$\varphi\leq\psi$, then $d_\varphi\leq d_\psi$.
\end{lemma}

\begin{proof}
Since $\varphi(x)\leq Q(x)=(\Down x,0)$, the second coordinate of
$\varphi(x)$ is $0$ and
$B_\varphi(x)\subseteq\Down x$.  The inequality
$\varphi(x)\leq(M,1)$ gives $B_\varphi(x)\subseteq M$.  This proves the stated inclusions.  If $B_\varphi(x)$ is nonempty, then it is a finite
directed set and therefore has a greatest element.  Hence
$d_\varphi$ is well defined, has image contained in $M$, and satisfies
$d_\varphi(x)\in B_\varphi(x)\subseteq\Down x.$
In particular $d_\varphi(x)\ll x$, so $d_\varphi\leq\id_X$.

The map $d_\varphi$ is monotone.  Indeed, $x\leq y$ and the monotonicity of
$\varphi$ imply
$B_\varphi(x)\subseteq B_\varphi(y)$, and hence
$d_\varphi(x)\leq d_\varphi(y)$.  To prove Scott continuity, let
$D\subseteq X$ be directed and put $x=\bigvee D$.  Since $\varphi$ is
Scott continuous and all its values have second coordinate $0$,
$B_\varphi(x)=\bigcup_{y\in D}B_\varphi(y).$
The element $d_\varphi(x)$ belongs to this union, so it belongs to
$B_\varphi(y_0)$ for some $y_0\in D$.  
Since $d_\varphi(x)\in B_\varphi(y_0)$ and
$d_\varphi(y_0)$ is the greatest element of $B_\varphi(y_0)$, we have
$d_\varphi(x)\leq d_\varphi(y_0)$.  Monotonicity and $y_0\leq x$
give the reverse inequality.
Hence $d_\varphi(y_0)=d_\varphi(x)$, and
monotonicity then gives $d_\varphi(y)=d_\varphi(x)$ for every
$y\in D$ above $y_0$.  It follows that
$d_\varphi\bigl(\bigvee D\bigr)=\bigvee_{y\in D}d_\varphi(y).$
Consequently $d_\varphi$ is a deflation.

Finally, if $\varphi\leq\psi$, then
$B_\varphi(x)\subseteq B_\psi(x)$ for every $x$.  Taking greatest elements
gives $d_\varphi(x)\leq d_\psi(x)$, and hence
$d_\varphi\leq d_\psi$.
\end{proof}

\begin{proposition}\label{prop:finite-cover-extraction}
Let $X$ be a domain.  Suppose that $[X\to E_X]$ is continuous and
that there are finitely many pairs
$u_i\ll v_i \ (1\leq i\leq m)$
such that
$X=\Up v_1\cup\cdots\cup\Up v_m.$
Then $X$ is an RB-domain.
\end{proposition}

\begin{proof}
Because $[X\to E_X]$ is continuous, $\Down Q$ is directed and has
supremum $Q$.  By Lemma~\ref{lem:requirements-waybelow}, each $S_{u_i,v_i}$ belongs to $\Down Q$, so directedness
gives some $\varphi_0\ll Q$ such that
$S_{u_i,v_i}\leq\varphi_0$ for all $1\leq i\leq m.$
Let
$\mathcal D=\{\varphi\ll Q:\varphi_0\leq\varphi\}.$
The set $\mathcal D$ is directed.  It is also cofinal in $\Down Q$: given
$\psi\ll Q$, directedness of $\Down Q$ supplies
$\varphi\ll Q$ above both $\psi$ and $\varphi_0$.

Fix $\varphi\in\mathcal D$.  Since $\varphi\ll Q$, one has
$\varphi\leq Q$. Since $Q \leq \bigvee_{M\in \Fin(X)}K_M$, there exists a finite
$M\subseteq X$ with $\varphi\leq K_M$.  Write
$\varphi(x)=(B_\varphi(x),0)$.  For each $x\in X$, choose $i$ with
$v_i\ll x$.  Then
$(\{u_i\},0)=S_{u_i,v_i}(x) \leq\varphi_0(x)\leq\varphi(x),$
so $u_i\in B_\varphi(x)$.  Thus every $B_\varphi(x)$ is nonempty, and
Lemma~\ref{lem:deflation-extraction} yields a deflation $d_\varphi$.
If $\varphi,\psi\in\mathcal D$, choose $\chi\in\mathcal D$ above both.
The last assertion of Lemma~\ref{lem:deflation-extraction} gives
$d_\varphi,d_\psi\leq d_\chi$.  Hence
$(d_\varphi)_{\varphi\in\mathcal D}$ is directed. 

Since all $d_\varphi$ are all deflations, it remains to compute their pointwise supremum.  
Let $x\in X$ and $u\ll x$. By interpolation, choose $v\in X$ such that
$u\ll v\ll x$. Since $S_{u,v}\ll Q$ and the family $\mathcal D$ is
cofinal below $Q$, there exists $\varphi\in\mathcal D$ with
$S_{u,v}\leq\varphi$. Evaluating at $x$ gives
\(
(\{u\},0)=S_{u,v}(x)\leq\varphi(x)=(B_\varphi(x),0).
\)
Hence $u\in B_\varphi(x)$, and therefore
$u\leq d_\varphi(x)\leq\bigvee_{\psi\in\mathcal D}d_\psi(x)$.
Since this holds for every $u\ll x$ and $X$ is continuous,
\(
x=\bigvee\Down x
\leq
\bigvee_{\psi\in\mathcal D}d_\psi(x).
\)
The reverse inequality follows from $d_\varphi\leq\id_X$ for every
$\varphi$.  Thus
$\bigvee_{\varphi\in\mathcal D}d_\varphi=\id_X,$
so $X$ is an RB-domain.
\end{proof}

\begin{corollary}\label{cor:pointed-test}
If $X$ is a pointed domain, then $[X\to E_X]$ is continuous if and only if
$X$ is an RB-domain.
\end{corollary}

\begin{proof}
The forward implication follows from
Proposition~\ref{prop:finite-cover-extraction} with the single pair
$u=v=\bot_X$, since $\bot_X\ll\bot_X$ and $\Up\bot_X=X$.  The reverse
implication is Proposition~\ref{prop:fact} (i).
\end{proof}

In the following, we will use $E_X$ combined with a three-point domain to construct the final test domain.
Let $V=\{\ell,r,\top\}$, ordered by $\ell<\top$ and $r<\top$,
with $\ell$ and $r$ incomparable.
The finite poset $V$ is an algebraic domain.  Its two minimal elements are
$\ell$ and $r$, while $\top$ lies above both.  Consequently, $V$ is not a
disjoint union of pointed domains.

Define
$A_X=E_X\times V.$
Since finite products of algebraic dcpos are algebraic,
Lemma~\ref{lem:EX-algebraic} shows that $A_X$ is an algebraic domain.  It is
not a disjoint union of pointed domains: the two distinct minimal elements
$((\varnothing,0),\ell) \ \text{and}\ ((\varnothing,0),r)$
both lie below $((\varnothing,0),\top)$.

\begin{theorem}\label{thm:universal-target}
Let $X$ be a domain. If $[X\to A_X]$ is continuous, then $X$ is an RB-domain.
\end{theorem}

\begin{proof}
Let $p_E:A_X\to E_X$ be the first-coordinate projection and define
$s_E:E_X\longrightarrow A_X, \ s_E(e)=(e,\ell).$
Then $p_E\circ s_E=\id_{E_X}$.
Define
\[
S:[X\to E_X]\longrightarrow[X\to A_X],
\quad
S(f)=s_E\circ f,
\]
and
\[
P:[X\to A_X]\longrightarrow[X\to E_X],
\quad
P(g)=p_E\circ g.
\]
Both maps are Scott continuous because directed suprema in the two
function spaces are computed pointwise and $s_E$ and $p_E$ preserve directed
suprema.  Since $p_E\circ s_E=\id_{E_X}$, for every $f\in[X\to E_X]$ we
have
\[
(P\circ S)(f)
 =p_E\circ s_E\circ f
 =f.
\]
Thus $P\circ S=\id_{[X\to E_X]}$, so $[X\to E_X]$ is a
Scott-continuous retract of $[X\to A_X]$.  Since $[X\to A_X]$ is
continuous and Scott-continuous retracts of continuous dcpos are
continuous, $[X\to E_X]$ is continuous.
By Proposition~\ref{prop:fact} (iv), the continuity of
$[X\to A_X]$ implies that either $X$ is compact in its Scott topology
or $A_X$ is a disjoint union of pointed domains.  The latter alternative
has already been ruled out.  Therefore $X$ is Scott compact.
For each $x\in X$, continuity and interpolation provide elements
$u\ll v\ll x$.  Thus the Scott-open sets
$\Up v \ (u\ll v\text{ for some }u\in X)$
cover $X$.  Scott compactness yields finitely many pairs
$u_i\ll v_i$ such that
$X=\Up v_1\cup\cdots\cup\Up v_m.$
Proposition~\ref{prop:finite-cover-extraction}, applied to the already
established continuity of $[X\to E_X]$, now shows that $X$ is an RB-domain.
\end{proof}

\begin{theorem}\label{thm:rb-characterization}
For every domain $X$, the following conditions are equivalent:
\begin{enumerate}[label=\textup{(\roman*)}]
\item $X$ is an RB-domain;
\item $[X\to E]$ is continuous for every domain $E$.
\end{enumerate}
Moreover, the algebraic domain $A_X$ constructed above has the property that
continuity of $[X\to A_X]$ already implies that $X$ is an RB-domain.
\end{theorem}

\begin{proof}
The implication (i)$\Rightarrow$(ii) is the theorem of
Miao--Luan--Liu--Li~\cite{MiaoEtAl2023}.  Conversely, assume (ii).  The
algebraic domain $A_X$ is a domain, so $[X\to A_X]$ is continuous;
Theorem~\ref{thm:universal-target} then yields (i).  The same theorem gives
the final, single-target assertion.
\end{proof}

\subsection{A uniform test domain for countably based domains}
\label{sec:countable-uniform-test}

The target $A_X$ is built from the underlying set and order of $X$.  For
countably based domains, this dependence can be removed by placing all finite
order patterns into one countable algebraic domain.

For a poset $P$, let $\operatorname{Idl}(P)$ denote the dcpo of nonempty
directed lower subsets of $P$, ordered by inclusion.
A \emph{finite rooted order} on $\mathbb N$ is a triple $\alpha=(F_\alpha,\preccurlyeq_\alpha,t_\alpha)$, where $F_\alpha$ is a nonempty finite subset of $\mathbb N$, $\preccurlyeq_\alpha$ is a partial order on $F_\alpha$, and $t_\alpha$ is its greatest element.  
Let $\mathcal R_\omega$ be the set of all finite rooted orders on $\mathbb N$.
Define a poset $P_\omega$ whose elements are $\bot_P$, the members of $\mathcal R_\omega$, and symbols $u_M$ for $M\in\Fin(\mathbb N)$.  The element $\bot_P$ is least, and the remaining order relations are
\[
 \alpha\leq\beta
 \quad\Longleftrightarrow\quad
 F_\alpha\subseteq F_\beta
 \text{ and }
 \preccurlyeq_\alpha\subseteq\preccurlyeq_\beta,
\]
\[
 \alpha\leq u_M
 \quad\Longleftrightarrow\quad
 F_\alpha\subseteq M,
 \qquad
 u_M\leq u_N
 \quad\Longleftrightarrow\quad
 M\subseteq N.
\]
There are no other comparabilities.

\begin{lemma}\label{lem:Pomega}
The poset $P_\omega$ is countable and directed.  Its ideal completion
$U_\omega=\operatorname{Idl}(P_\omega)$ is a pointed countably based
algebraic domain.
\end{lemma}

\begin{proof}
The displayed clauses define a partial order.  For example, if
$\alpha\leq\beta\leq u_M$, then
$F_\alpha\subseteq F_\beta\subseteq M$, so $\alpha\leq u_M$.
Antisymmetry among rooted orders follows because a finite partial order has
at most one greatest element.  Every finite subset of $P_\omega$ has an
upper bound $u_M$, where $M$ is the union of all supports and all finite
sets occurring in that subset.  Hence $P_\omega$ is directed.  It is
countable because it is built from finite subsets and finite relations on
$\mathbb N$.  The principal ideals form a countable compact basis of its
ideal completion, which is pointed because $P_\omega$ has a least element.
\end{proof}

Let $X$ be a countably based domain. Choose a countable basis
\(
B=\{b_n:n\in L\},
\)
where $L\subseteq\mathbb N$ and the elements $b_n$ are pairwise distinct.
A rooted order
$\alpha=(F_\alpha,\preccurlyeq_\alpha,t_\alpha)$ is called
$B$-admissible if $F_\alpha\subseteq L$ and
\(
n\preccurlyeq_\alpha m
\ \Longrightarrow\
b_n\leq b_m
\)
for all $n,m\in F_\alpha$.
For each $x\in X$, define
\[
Q_B(x)
=
\{\bot_P\}
\cup
\bigl\{
\alpha\in\mathcal R_\omega:
\alpha\text{ is $B$-admissible and }b_{t_\alpha}\ll x
\bigr\}.
\]
For $p\in P_\omega$, write
\(
\widehat p
=
\downarrow_{P_\omega}p
=
\{q\in P_\omega:q\leq p\}.
\)
Then $\widehat p$ is the principal ideal generated by $p$, and hence a
compact element of
$U_\omega=\operatorname{Idl}(P_\omega)$.

\begin{lemma}\label{lem:QB-uniform}
Let $X$ be a countably based domain. For every $x\in X$, the set $Q_B(x)$ is an ideal of $P_\omega$, and
$Q_B:X\to U_\omega$ is Scott continuous.
\end{lemma}

\begin{proof}
Suppose that $\alpha\leq\beta$ and $\beta\in Q_B(x)$.  Then $\alpha$ is
$B$-admissible.  Since $t_\beta$ is greatest in $F_\beta$ and
$t_\alpha\in F_\beta$, one has
$b_{t_\alpha}\leq b_{t_\beta}\ll x$, so $\alpha\in Q_B(x)$.
Thus $Q_B(x)$ is lower.

Let $\alpha,\beta\in Q_B(x)$.  Since $B\cap\Down x$ is directed, there is
$b_r\in B\cap\Down x$ above both $b_{t_\alpha}$ and $b_{t_\beta}$.
On $F_\alpha\cup F_\beta\cup\{r\}$ define
\(
n\preccurlyeq m
\ \Longleftrightarrow\ 
b_n\leq_X b_m.
\)
Then $r$ is the greatest element of this finite order.
With $r$ as greatest element, this gives a $B$-admissible rooted order
$\gamma\in Q_B(x)$ above $\alpha$ and $\beta$.  Hence $Q_B(x)$ is directed.

The map $Q_B$ is monotone.  Moreover, for a rooted order $\alpha$, we have
\[
 Q_B^{-1}(\up\widehat\alpha)=
 \begin{cases}
  \Up b_{t_\alpha},&\alpha\text{ is $B$-admissible},\\
  \varnothing,&\text{otherwise},
 \end{cases}
\]
while $Q_B^{-1}(\up\widehat{u_M})=\varnothing$ and
$Q_B^{-1}(\up\widehat{\bot_P})=X$.  These are Scott open, and the
principal ideals form a compact basis of $U_\omega$.  Therefore $Q_B$ is
Scott continuous.
\end{proof}

For $M\in\Fin(\mathbb N)$, let $K_M:X\to U_\omega$ be constant with value
$\widehat{u_M}$.  The maps $K_M$ form a directed family whose supremum is
the constant top map.  Consequently, every $\varphi\ll Q_B$ satisfies
$\varphi\leq K_M$ for some finite $M$.  If $\varphi\leq Q_B$, put
\[
 L_\varphi(x)=
 \{b_{t_\alpha}:\alpha\in\varphi(x)\cap\mathcal R_\omega\}.
\]
Every rooted order occurring in $\varphi(x)$ has support contained in $M$,
so $L_\varphi(x)$ is a finite subset of
$\{b_n:n\in M\cap L\}\cap\Down x$.  
It is directed: a common upper bound of two rooted orders in the ideal $\varphi(x)\subseteq Q_B(x)$ is again a
rooted order, and its greatest label gives an upper bound of the two
corresponding basis elements.  Whenever $L_\varphi(x)$ is nonempty, it
therefore has a greatest element.  The proof of
Lemma~\ref{lem:deflation-extraction}, with $L_\varphi(x)$ in place of
$B_\varphi(x)$, shows that
$d_\varphi(x)=\max L_\varphi(x)$ is a finite-image deflation.  It also
shows that $\varphi\leq\psi$ implies $d_\varphi\leq d_\psi$ whenever both
maps are defined.

For $n\in L$, let $\alpha_n$ be the rooted order on the singleton $\{n\}$.
If $b_n\ll v$, define $S_{n,v}:X\to U_\omega$ by
\[
 S_{n,v}(x)=
 \begin{cases}
  \widehat{\alpha_n},&v\ll x,\\
  \widehat{\bot_P},&v\not\ll x.
 \end{cases}
\]
Since $Q_B^{-1}(\up\widehat{\alpha_n})=\Up b_n$ and
$\Up v\ll_{\mathcal O(X)}\Up b_n$, Lemma~\ref{lem:standard-step} gives
$S_{n,v}\ll Q_B$.

\begin{proposition}\label{prop:uniform-finite-cover}
Let $X$ be a countably based domain. Suppose that $[X\to U_\omega]$ is continuous and that
$X=\Up v_1\cup\cdots\cup\Up v_m$ for some
$b_{n_i}\ll v_i$, where $n_i\in L$.  Then $X$ is an RB-domain.
\end{proposition}

\begin{proof}
Choose $\varphi_0\ll Q_B$ above the finitely many maps $S_{n_i,v_i}$, and
let $\mathcal D$ be the set of all $\varphi\ll Q_B$ above $\varphi_0$.
This is a directed cofinal subset of $\Down Q_B$.  For every
$\varphi\in\mathcal D$, the finite bound above and the covering condition
make each $L_\varphi(x)$ nonempty, so $d_\varphi$ is a finite-image
deflation.  These deflations form a directed family.

Let $u\ll x$.  By the basis property and interpolation, choose $n\in L$
and $v\in X$ with $u\leq b_n\ll v\ll x$.  Cofinality gives
$\varphi\in\mathcal D$ with $S_{n,v}\leq\varphi$, and then
$u\leq b_n\leq d_\varphi(x)$.  
For every $u\ll x$, some $\varphi$ satisfies
$u\leq d_\varphi(x)$.  Since $X$ is continuous,
\(
x=\bigvee\Down x
\leq
\bigvee_\varphi d_\varphi(x).
\)
The reverse inequality follows from $d_\varphi\leq\id_X$.
Hence $X$ is an RB-domain.
\end{proof}

Let $V$ be the three-element algebraic domain used in the preceding
subsection and put $A_\omega=U_\omega\times V$.  This is a countably based
algebraic domain.  It is not a disjoint union of pointed domains, because
its two minimal elements $(\widehat{\bot_P},\ell)$ and
$(\widehat{\bot_P},r)$ have the common upper bound
$(\widehat{\bot_P},\top)$.

\begin{theorem}\label{thm:uniform-countable-test}
For every countably based domain $X$,
\[
 X\in\RB
 \quad\Longleftrightarrow\quad
 [X\to A_\omega]\text{ is continuous}.
\]
If $X$ is pointed, then
\[
 X\in\RB
 \quad\Longleftrightarrow\quad
 [X\to U_\omega]\text{ is continuous}.
\]
\end{theorem}

\begin{proof}
If $X$ is an RB-domain, both function spaces are continuous by Proposition~\ref{prop:fact} (i).  Conversely, suppose that $[X\to A_\omega]$ is
continuous.  Projection onto the first coordinate and the section
$e\mapsto(e,\ell)$ show, exactly as in the proof of
Theorem~\ref{thm:universal-target}, that $[X\to U_\omega]$ is a
Scott-continuous retract of $[X\to A_\omega]$, and hence is continuous.
Proposition~\ref{prop:fact} (iv) implies that $X$ is Scott compact, because
$A_\omega$ is not a disjoint union of pointed domains.

For every $x\in X$, choose $b_n\in B\cap\Down x$ and then choose $v$ with
$b_n\ll v\ll x$.  The corresponding sets $\Up v$ cover $X$, so Scott
compactness gives a finite subcover.  Proposition~\ref{prop:uniform-finite-cover}
now yields $X\in\RB$.

If $X$ is pointed, enlarge the chosen basis by $\bot_X$ if necessary.  The
single choice $b_n=v=\bot_X$ satisfies $b_n\ll v$ and $\Up v=X$.
Proposition~\ref{prop:uniform-finite-cover} therefore gives the converse
from the continuity of $[X\to U_\omega]$.
\end{proof}

\section{Function spaces of FS-domains}\label{sec:fs}

The RB characterization has two natural FS analogues: one may ask whether an
FS-domain preserves continuity of the function space for all targets, or whether FS-domains can
be detected from the class of function spaces it generates.  Of course, if FS-domains are not identical to RB-domains, then the first analogue is false. We will use an example to show that.
The second becomes correct after strengthening the conclusion from continuity to
FS-membership.  

\subsection{Failure of the continuity analogue}\label{sec:disk}
 Let $\Disc$ be the dcpo of all closed disks in the Euclidean plane, together
with the whole plane as its least element, ordered by reverse inclusion.  Its
FS property is recorded in
\cite[Section~4.2.2, example following Proposition~4.2.12]{AbramskyJung1994};
it also follows from Lawson's formal-ball theorem \cite{Lawson2008}.  Using
\cite[Proposition~5.3]{ChenKouLyu2026}, we give a direct function-space
obstruction showing that $[\Disc\to E_{\Disc}]$ is not continuous.

Represent a non-bottom disk $\overline B(\mathbf z,r)$ by
$\mathbf x=(-r,\mathbf z)\in H:=(-\infty,0]\times\mathbb R^2$.
Throughout this subsection, $\|\cdot\|$ denotes the Euclidean norm and
$\mathbf v\otimes\mathbf w$ denotes the matrix
$\mathbf v\mathbf w^T$.  Let
$C=\{(t,\mathbf z)\in\mathbb R\times\mathbb R^2:
 t\geq\|\mathbf z\|\}$
be the three-dimensional Lorentz cone, and write
$\mathbf x\leq_C\mathbf y$ when $\mathbf y-\mathbf x\in C$.
A map between subsets of $\mathbb R^3$ is called $C$-monotone if it is
order preserving with respect to $\leq_C$.  Under the representation above,
reverse inclusion of disks is precisely the
order $\leq_C$, and for non-bottom elements
\(
  \mathbf x\ll\mathbf y
  \ \Longleftrightarrow \ 
  \mathbf y-\mathbf x\in\operatorname{int}C
\)
\cite[Section~3]{ChenKouLyu2026}.  Put
$J=\operatorname{diag}(1,-1,-1)$ and
$\Lambda(A)=\operatorname{tr}(JA)$.  If
$\mathbf v=(v_0,\mathbf v')$ and
$\mathbf w=(w_0,\mathbf w')$ belong to $C$, then
\[
  \Lambda(\mathbf v\otimes\mathbf w)
  =\mathbf w^{T}J\mathbf v
  =v_0w_0-\mathbf v'\mathbin{\cdot}\mathbf w'
  \geq0,
\]
whereas
$\Lambda(I_3)=-1.$
We use the following scalar consequence of
\cite[Proposition~5.3]{ChenKouLyu2026}.

\begin{proposition}\label{prop:disk-scalar-flux}
Let $\Omega\subseteq\mathbb R^3$ be an open rectangular box, let
$q:\Omega\to\mathbb R^3$ be a finite-image $C$-monotone map, and let
$0\leq\psi\in C_c^\infty(\Omega)$.  Then $q$ is Lebesgue measurable,
the following matrix integral is well defined coordinatewise, and
\[
  \Lambda\!\left(
    -\int_\Omega q(\mathbf x)\otimes\nabla\psi(\mathbf x)\,d\mathbf x
  \right)\geq0.
\]
\end{proposition}

\begin{proof}
Proposition \cite[Proposition~5.3]{ChenKouLyu2026} gives measurability and places the tested matrix in the
closed convex cone generated by
$\mathbf v\otimes\mathbf w$ with $\mathbf v,\mathbf w\in C$.  The
preceding calculation shows that $\Lambda$ is nonnegative on every generator
and hence on that cone.
\end{proof}

Fix
$P=[-6,-4]\times[-1,1]^2, \ \Omega=(-6,-4)\times(-1,1)^2,$
and choose $0\leq\psi\in C_c^\infty(\Omega)$ with
$\int_\Omega\psi=1$.  Set
\(
  L_\psi=\int_\Omega\|\nabla\psi(\mathbf x)\|\,d\mathbf x
\)
and choose $\eta>0$ such that $\eta L_\psi<1$.  Let
$\mathbf e=(1,0,0)$.

\begin{lemma}\label{lem:disk-order-interval}
If $\mathbf x-\eta\mathbf e\leq_C\mathbf y\leq_C\mathbf x$, then
$\|\mathbf x-\mathbf y\|\leq\eta.$
\end{lemma}

\begin{proof}
Write $\mathbf d=\mathbf x-\mathbf y=(d_0,\mathbf d')$.
The two order inequalities say that
$\mathbf d\in C$ and $\eta\mathbf e-\mathbf d\in C$.  Hence
$d_0\geq\|\mathbf d'\|$ and $\eta-d_0\geq\|\mathbf d'\|$.
Thus $0\leq d_0\leq\eta$ and
$\|\mathbf d'\|\leq\min\{d_0,\eta-d_0\}$.  If
$d_0\leq\eta/2$, then
$\|\mathbf d\|^2\leq2d_0^2\leq\eta^2/2$.  If
$d_0\geq\eta/2$, then
$\|\mathbf d\|^2 \leq d_0^2+(\eta-d_0)^2 \leq\eta^2.$
In either case $\|\mathbf d\|\leq\eta$.
\end{proof}

\begin{proposition}\label{prop:disk-fixed-scale}
There is no finite-image $C$-monotone map $s:P\to H$ satisfying
$\mathbf x-\eta\mathbf e\leq_C s(\mathbf x)\leq_C\mathbf x \ (\mathbf x\in P).$
\end{proposition}

\begin{proof}
Suppose that such an $s$ exists and put $q=s|_\Omega$.  By
Proposition~\ref{prop:disk-scalar-flux},
\(
  \Lambda\!\left(
    -\int_\Omega q\otimes\nabla\psi
  \right)\geq0.
\)
Entrywise integration by parts gives
\(
  -\int_\Omega \mathbf x\otimes\nabla\psi(\mathbf x)\,d\mathbf x
  =I_3,
\)
because
$-\int_\Omega x_i\partial_j\psi=\delta_{ij}\int_\Omega\psi
=\delta_{ij}$.
By Lemma~\ref{lem:disk-order-interval},
$\sup_{\Omega}\|q(\mathbf x)-\mathbf x\|\leq\eta$. Since $J$ only
changes the signs of the last two coordinates,
$\|J\mathbf z\|=\|\mathbf z\|$ for every $\mathbf z\in\mathbb R^3$.
We have
\begin{align*}
 \left|
  \Lambda\!\left(-\int_\Omega q\otimes\nabla\psi\right)
  -\Lambda(I_3)
  \right| 
 &=
 \left|
   \int_\Omega
   \nabla\psi(\mathbf x)^TJ
   \bigl(\mathbf x-q(\mathbf x)\bigr)\,d\mathbf x
 \right| 
 \leq
 \sup_{\Omega}\|q(\mathbf x)-\mathbf x\|
 \int_\Omega\|\nabla\psi(\mathbf x)\|\,d\mathbf x
 \leq\eta L_\psi<1.
\end{align*}
Together with $\Lambda(I_3)=-1$, this implies
\(
  \Lambda\!\left(-\int_\Omega q\otimes\nabla\psi\right)<0,
\)
contradicting the nonnegativity established above.
\end{proof}

For each $\mathbf x\in P$, put
\(
  \mathbf u_{\mathbf x}
  =\mathbf x-\frac{3\eta}{4}\mathbf e,
  \
  \mathbf v_{\mathbf x}
  =\mathbf x-\frac{\eta}{2}\mathbf e.
\)
Then
$\mathbf u_{\mathbf x}\ll\mathbf v_{\mathbf x}\ll\mathbf x$ by
the characterization of the way-below relation above.  Moreover, the set
\(
  N_{\mathbf x}
  =\left\{\mathbf y\in H:
    \mathbf y-\mathbf v_{\mathbf x}\in\operatorname{int}C,
    \
    \mathbf u_{\mathbf x}-\mathbf y+\eta\mathbf e
       \in\operatorname{int}C
  \right\}
\)
is a Euclidean neighborhood of $\mathbf x$.  Indeed, at
$\mathbf y=\mathbf x$ the two displayed vectors are respectively
$(\eta/2)\mathbf e$ and $(\eta/4)\mathbf e$, both of which belong to
$\operatorname{int}C$.  Compactness of $P$ therefore gives
$\mathbf x_1,\ldots,\mathbf x_m\in P$ such that
$P\subseteq\bigcup_{j=1}^mN_{\mathbf x_j}$.  Write
$\mathbf u_j=\mathbf u_{\mathbf x_j}, \ \mathbf v_j=\mathbf v_{\mathbf x_j}.$
Then, for every $\mathbf y\in P$, some $j$ satisfies
$\mathbf v_j\ll\mathbf y, \ \mathbf y-\eta\mathbf e\leq_C\mathbf u_j.$
Let
$Q_{\Disc}:\Disc\to E_{\Disc}, \ Q_{\Disc}(x)=(\Down x,0),$
be the canonical map introduced above.  For
$j=1,\ldots,m$, let
$S_j=S_{\mathbf u_j,\mathbf v_j},$
where $S_{u,v}$ is the step map defined above.  Since
$\mathbf u_j\ll\mathbf v_j$, Lemma~\ref{lem:requirements-waybelow}
gives
$S_j\ll Q_{\Disc} \ (1\leq j\leq m).$
\begin{theorem}\label{thm:disk-direct-noncontinuity}
The finite family $S_1,\ldots,S_m$ has no common upper bound in
$\Down Q_{\Disc}$.  
Consequently,
\(
  [\Disc\to E_{\Disc}]\)
  is not continuous.
\end{theorem}

\begin{proof}
Assume, to the contrary, that there is
$\varphi\ll Q_{\Disc}$ satisfying
$S_j\leq\varphi \quad(1\leq j\leq m).$
Since the constant maps $K_M$, $M\in\Fin(\Disc)$, form a directed
family whose supremum is the constant top map, and since
$Q_{\Disc}$ lies below that supremum, the relation
$\varphi\ll Q_{\Disc}$ yields a finite set $M\subseteq\Disc$ such that
$\varphi\leq K_M.$
Also $\varphi\leq Q_{\Disc}$.  Hence, for every $x\in\Disc$, there is
$B_\varphi(x)\in\mathcal{D}(\Disc)$ such that
$\varphi(x)=(B_\varphi(x),0), \ B_\varphi(x)\subseteq M\cap\Down x.$
Fix $\mathbf y\in P$ and choose $j$ so that $\mathbf v_j\ll\mathbf y$ and $\mathbf y-\eta\mathbf e\leq_C\mathbf u_j$.  Since
$\mathbf v_j\ll\mathbf y$, the definition of $S_j$ and the assumption $S_j\leq\varphi$ give
$(\{\mathbf u_j\},0)=S_j(\mathbf y) \leq\varphi(\mathbf y).$
Thus $\mathbf u_j\in B_\varphi(\mathbf y)$.  In particular,
$B_\varphi(\mathbf y)$ is a nonempty finite directed set.  It therefore
has a greatest element; define
$s(\mathbf y)=\max B_\varphi(\mathbf y) \ (\mathbf y\in P).$
Since $B_\varphi(\mathbf y)$ contains the non-bottom element
$\mathbf u_j$, its greatest element is non-bottom and hence belongs to
$H$.  The image of $s$ is contained in the finite set $M$.  If
$\mathbf x\leq_C\mathbf y$ in $P$, monotonicity of $\varphi$ gives
$B_\varphi(\mathbf x)\subseteq B_\varphi(\mathbf y)$, and hence
$s(\mathbf x)\leq_C s(\mathbf y)$.  Thus $s$ is $C$-monotone.

For the lower estimate, choose $j$ so that $\mathbf v_j\ll\mathbf y$ and $\mathbf y-\eta\mathbf e\leq_C\mathbf u_j$.  Since
$\mathbf u_j\in B_\varphi(\mathbf y)$,
$\mathbf y-\eta\mathbf e \leq_C\mathbf u_j \leq_C s(\mathbf y).$
For the upper estimate, the inclusion $B_\varphi(\mathbf y)\subseteq\Down\mathbf y$ gives
$s(\mathbf y)\in\Down\mathbf y$, so
$s(\mathbf y)\ll\mathbf y$ and therefore
$s(\mathbf y)\leq_C\mathbf y$.  We have therefore
constructed a finite-image $C$-monotone map $s:P\to H$ satisfying the two inequalities in Proposition~\ref{prop:disk-fixed-scale}, contradicting
Proposition~\ref{prop:disk-fixed-scale}.

Thus the maps $S_1,\ldots,S_m$, each way below $Q_{\Disc}$, have no common upper
bound in $\Down Q_{\Disc}$.  If $[\Disc\to E_{\Disc}]$ were continuous,
then $\Down Q_{\Disc}$ would be directed, and a finite subset of it would
have a common upper bound in it.  Hence the function space is not
continuous.
\end{proof}

\begin{corollary}\label{thm:fs-counterexample}
There is a pointed algebraic domain $E$ such that
$\Disc\in\FS$ and $[\Disc\to E]$ is not continuous.
\end{corollary}

\begin{proof}
By Lemma~\ref{lem:EX-algebraic}, $E_{\Disc}$ is a pointed algebraic
domain.  The disk domain belongs to $\FS$ by the classical results cited at
the beginning of this section, and
Theorem~\ref{thm:disk-direct-noncontinuity} shows that
$[\Disc\to E_{\Disc}]$ is not continuous.  Thus one may take
$E=E_{\Disc}$.
\end{proof}

The preceding counterexample concerns only continuity of the function
space.  If the function space is itself required to be an FS-domain, the
situation is simpler and follows formally from Cartesian closure and retract
closure.

\begin{remark}\label{rem:fs-valued-characterizations}
For every domain $D$,
\[
D\in\FS
\quad\Longleftrightarrow\quad
[D\to D]\in\FS.
\]
Indeed, the forward implication follows from the Cartesian closure of
$\FS$.  Conversely, fix $d_0\in D$.  The constant-map embedding
$\const:D\to[D\to D]$ and the evaluation map
$\ev_{d_0}:[D\to D]\to D$ satisfy
$\ev_{d_0}\circ\const=\id_D$.  Thus $D$ is a Scott-continuous retract of
$[D\to D]$, and the converse follows from the retract closure of
FS-domains.

More generally, let $S$ be any fixed FS-domain.  For every domain $E$, the
following conditions are equivalent:
\[
E\in\FS,
\qquad
[S\to E]\in\FS,
\qquad
[D\to E]\in\FS\text{ for every }D\in\FS.
\]
The only nonformal implication is again the retract argument.  If
$[S\to E]$ is an FS-domain, choose $s_0\in S$.  The constant-map embedding
$E\to[S\to E]$ and evaluation at $s_0$ exhibit $E$ as a
Scott-continuous retract of $[S\to E]$.
\end{remark}

The preceding results leave the following two natural questions open.

\begin{question}\label{q:coherent-source-fs-target}
Let $D$ be a coherent domain and let $E$ be a pointed FS-domain.  Is the
function space $[D\to E]$ necessarily continuous?
\end{question}


\begin{question}\label{q:universal-fs-targets}
Suppose that $E$ is a domain such that $[D\to E]$ is continuous for every
FS-domain $D$.  Must $E$ be a $\ll$-separating domain?
\end{question}


\end{document}